\documentclass[a4paper,10pt]{article}
\usepackage{mathtext}
\usepackage[T1,T2A]{fontenc}
\usepackage[cp1251]{inputenc}
\usepackage[english]{babel}
\usepackage{amsmath}
\usepackage{amsfonts}
\usepackage{amssymb}
\usepackage{mathrsfs}
\usepackage{amsthm}
\usepackage{enumerate}
\usepackage{graphicx}

\usepackage{color}
\usepackage{euscript}

\usepackage{cite}

\newtheorem{Le}{Lemma}[section]

\newtheorem{St}[Le]{Proposition}
\newtheorem{Th}{Theorem}[section]

\newtheorem{Rem}[Le]{Remark}

\numberwithin{equation}{section}

\newcommand{\R}{\mathbb{R}}
\newcommand{\Co}{\mathbb{C}}

\newcommand{\eps}{\varepsilon}

\newcommand{\B}{\mathbb{B}}

\newcommand{\eq}[1]{\begin{equation}{#1}\end{equation}}
\newcommand{\mlt}[1]{\begin{multline}{#1}\end{multline}}
\newcommand{\alg}[1]{\begin{align}{#1}\end{align}}

\newcommand{\set}[2]{\{{#1}\mid{#2}\}}
\newcommand{\Set}[2]{\Big\{{#1}\,\Big|\;{#2}\Big\}}
\newcommand{\scalprod}[2]{\langle{#1},{#2}\rangle}
\newcommand{\Scalprod}[2]{\Big\langle{#1},{#2}\Big\rangle}
\newcommand{\fdot}{\,\cdot\,}

\DeclareMathOperator{\BV}{BV}
\DeclareMathOperator{\I}{I}
\DeclareMathOperator{\J}{J}
\DeclareMathOperator{\dist}{dist}
\DeclareMathOperator{\clos}{clos}

\DeclareMathOperator{\Dil}{Dil}

\newcommand{\vp}{\mathrm{v.p.}}

\newcommand{\Image}{\mathrm{Im}}

\newcommand{\Me}{\boldsymbol{\mathrm{M}}}

\title{Weakly canceling operators and singular integrals}
\author{Dmitriy Stolyarov\thanks{Supported by the Russian Science Foundation grant N. 19-71-30002.}\;\thanks{I would like to thank Bogdan Raita for introducing me to his results, fruitful discussions, and exposition advice.}}
\begin{document}
\maketitle
\begin{abstract}
We suggest an elementary Harmonic Analysis approach to canceling and weakly canceling differential operators, which allows to extend these notions to anisotropic setting and also replace differential operators with Fourier multiplies with mild smoothness regularity. In this more general setting of anisotropic Fourier multipliers, we prove the inequality~$\|f\|_{L_{\infty}} \lesssim \|Af\|_{L_1}$ if~$A$ is a weakly canceling operator of order~$d$ and the inequality~$\|f\|_{L_2} \lesssim \|Af\|_{L_1}$ if~$A$ is a canceling operator  of order~$\frac{d}{2}$, provided~$f$ is a function in~$d$ variables. 
\end{abstract}

\section{Introduction}
Consider two similar inequalities for smooth compactly supported functions in two variables:
\eq{\label{Basic}
\|f\|_{L_2} \lesssim \|\nabla f\|_{L_1(\R^2)};\qquad \|\nabla f\|_{L_2} \lesssim \|\Delta f\|_{L_1(\R^2)}.
} 
Both inequalities are dilation invariant. The first one was proved by Gagliardo and Nirenberg in late 50s and the second one is false. The reason for this phenomenon is that~$\Delta f$ may be almost an arbitrary~$L_1$ function whereas the vectorial function~$\nabla f$ possesses hidden cancellations that distinguish it from an arbitrary~$L_1$ function. In~\cite{VanSchaftingen2013}, Van Schaftingen gave a simple necessary and sufficient condition on a differential operator on the right hand side of such type inequalities for the inequality to be true (we refer the reader to the early papers~\cite{BourgainBrezis2002}, \cite{BourgainBrezis2007}, \cite{VanSchaftingen2004}, \cite{VanSchaftingen2008}, and the survey~\cite{VanSchaftingen2014} for historical information about the preceding development of the subject). Let~$A$ be a differential operator that maps~$V$-valued functions into~$E$-valued functions, where~$V$ and~$E$ are finite dimensional spaces over~$\Co$. We will denote the symbol of~$A$ (which is a polynomial in~$d$ variables with values in the space of linear operators from~$V$ to~$E$) by the same letter~$A$. Assume~$A$ is elliptic, which in this setting means that the kernel of the linear operator~$A(\xi)$ is trivial for any~$\xi \in \R^d\setminus \{0\}$, and homogeneous of order~$m$. The operator~$A$ is called canceling provided
\eq{\label{Cancel}
\bigcap\limits_{\zeta \in S^{d-1}} \Image (A(\zeta)) = \{0\};
}
here~$S^{d-1}$ denotes the unit sphere. Note that, in particular, this is possible when~$\dim E > \dim V$ only. In other words, the corresponding system of partial differential equations~$A f = u$ is overdetermined. In~\cite{VanSchaftingen2013}, it was proved that an elliptic homogeneous of order~$m$ operator~$A$ is canceling if and only if
\eq{\label{canceling1}
\|\nabla^{m-1} f\|_{L_{\frac{d}{d-1}}} \lesssim \|A f\|_{L_1(\R^d,E)}.
}
Here and in what follows, the notation~$A \lesssim B$ means~$A \leq CB$, where the constant~$C$ may be chosen uniformly with respect to parameters that are clear from the context. For example, the constant in the inequality above  should be independent of~$f$. The space~$L_1(\R^d, E)$ is the space of summable~$E$-valued functions on~$\R^d$ equipped with the norm
\eq{
\|g\|_{L_1(\R^d,E)} = \int\limits_{\R^d}\|g(x)\|_E\,dx;
}
the particular choice of the norm on~$E$ is not important. By the classical Sobolev embedding, inequality~\eqref{canceling1} implies
\eq{\label{canceling2}
\|\nabla^{m-n} f\|_{L_{\frac{d}{d-n}}} \lesssim \|A f\|_{L_1(\R^d,E)}
}
whenever~$1 \leq n \leq m$ and~$n < d$. The case~$n=d$ (with additional assumption~$m \geq d$) is special; we postpone the discussion of this case. 

Van Schaftingen's theorem may be restated in another way. Consider the space
\alg{\label{SpacesPolynom}
W_1^{A} = \clos_{L_1(\R^d,E)}\set{Af}{f \in C_0^{\infty}(\R^d,V)},
}
which may be thought of as  a generalization of the classical homogeneous Sobolev space~$\dot{W}_1^1$: the space~$W_1^A$ becomes~$\dot{W}^1_1$ in the case~$A = \nabla$.  The space~\eqref{SpacesPolynom} inherits the norm from~$L_1(\R^d,E)$. In the case~$n < d$, inequality~\eqref{canceling2} is equivalent via the Calder\'on--Zygmund theory to
\eq{\label{canceling2}
\|\I_{n}[g]\|_{L_{\frac{d}{d-n}}} \lesssim \|g\|_{W_1^A};
}
where~$\I_\alpha$ is the Riesz potential of order~$\alpha \in (0,d)$, i.e. the Fourier multiplier with the symbol~$|\fdot|^{-\alpha}$. The space~\eqref{SpacesPolynom} has a good description in terms of the Fourier transform:
\eq{\label{Intermediatespace}
W_1^{A} = \Set{g\in L_1(\R^d,E)}{\forall \xi \in \R^d\setminus \{0\}\quad \hat{g}(\xi) \subset \Image(A(\xi))}.
}
We note that, by ellipticity of~$A$, the linear spaces~$A(\xi)$ have dimension~$\dim V$ that does not depend on~$\xi$. The paper~\cite{Raita2018} extends Van Schaftingen's theory to the setting of differential operators of constant rank; this class of operators is wider than the class of elliptic operators, in this case the inequality~\eqref{canceling2} is also true provided one modifies the function on the left hand side in an appropriate way. The representation~\eqref{Intermediatespace} is suitable for generalizations.  We borrow the notation from~\cite{ASW2018}.

Let us identify~$E$ with~$\Co^\ell$, where $\dim E = \ell$. By~$G(\ell,k)$ we denote the complex Grassmannian of~$k$-dimensional subspaces of~$\Co^\ell$, i.e. the collection of all~$k$-dimensional linear subspaces of~$\Co^\ell$. This set has natural metric space structure. We may identify~$L\in G(\ell,k)$ with the orthogonal projection~$\pi_L$ of~$\Co^\ell$ onto~$L$, thus embedding~$G(\ell,k)$ into the space of~$\ell\times\ell$ matrices (which equals~$\Co^{\ell^2}$). The set~$G(\ell,k)$ inherits Euclidean metric from the ambient space: for any~$L_1, L_2 \in G(\ell,k)$, the quantity~$\|\pi_{L_1} - \pi_{L_2}\|_{\mathrm{HS}}$ plays the role of the distance between~$L_1$ and~$L_2$; the norm~$\|\cdot\|_{\mathrm{HS}}$ is the Hilbert--Schmidt norm on the space of~$\ell \times \ell$ complex matrices. Fix~$\ell$ and~$k$ and let~$\Omega \colon S^{d-1}\to G(\ell,k)$ be a continuous function, define
\eq{\label{IsotropicOmega}
W_1^\Omega = \Set{g\in L_1(\R^d,\Co^\ell)}{\forall \xi \in \R^d\setminus \{0\}\quad \hat{g}(\xi) \subset \Omega\Big(\frac{\xi}{|\xi|}\Big)}.
}
By a H\"older continuous mapping between metric spaces, we mean a function~$F\colon M \to N$ such that
\eq{
\dist_N (F(x),F(y)) \lesssim \dist_M (x,y)^{\alpha}
} 
from some uniform~$\alpha > 0$ and all~$x,y\in M$ such that~$ \dist_M (x,y) \leq 1$.
\begin{Th}\label{L2Theorem}
Let~$\Omega$ be H\"older continuous and canceling in the sense that
\eq{\label{CancelBundle}
\bigcap_{\zeta \in S^{d-1}}\Omega(\zeta) = \{0\}.
}
Then\textup,~$\I_{\frac{d}{2}}$ maps~$W_1^\Omega$ to~$L_2$ continuously.
\end{Th}
It was conjectured in~\cite{ASW2018} that if~$\Omega$ is smooth and canceling in the same sense as above, then~$\I_{\alpha}$ maps~$W_1^\Omega$ to~$L_{\frac{d}{d-\alpha}}$ continuously. Theorem~\ref{L2Theorem} confirms the conjecture for~$\alpha  \geq \frac{d}{2}$. In particular, Theorem~\ref{L2Theorem} provides an alternative proof of Van Schaftingen's result~\eqref{canceling1} in dimension~$d=2$. Let us now return to the case~$\alpha = d$ and the language of differential operators for a moment. It appears that the cancellation condition~\eqref{Cancel} is only sufficient for the inequality
\eq{\label{canceling3}
\|\nabla^{m-d} f\|_{L_{\infty}} \lesssim \|A f\|_{L_1(\R^d,E)}.
}
In~\cite{Raita2019}, Raita gave necessary and sufficient conditions on~$A$ for this inequality to be true, the corresponding (elliptic) operators are called weakly canceling (similar effect is also present in the setting of constant rank operators, see~\cite{Raita2018}). Note that in this case the Calder\'on--Zygmund theory is not applicable directly, so the result depends not only on the collection~$\{\Image(A(\zeta))\}_{\zeta \in S^{d-1}}$, but on the structure of~$A$ as well. Now let us return to the Fourier transform notation. 

Consider the set~$\omega$ of pairs~$(\zeta,v)$, where~$\zeta \in S^{d-1}$ and~$v\in \Omega(\zeta)$ and a function~$B \colon \omega \to \mathbb{C}$. The function~$B$ is assumed to be linear with respect to the second variable and H\"older continuous as a mapping of~$\omega \subset S^{d-1}\times \Co^\ell$ to~$\Co$. In other words,~$B(\zeta,\cdot)$ is a linear functional on~$\Omega(\zeta)$ that depends in a H\"older continuous way on~$\zeta$. We prefer to write~$B(\zeta)[v]$ instead of~$B(\zeta,v)$ since the two arguments have different nature. Consider the vectorial Fourier multiplier~$\B$ acting on functions~$g \in W_1^\Omega$ by the rule
\eq{\label{Functional}
\B[g] = \mathcal{F}^{-1}\Big[|\xi|^{-d}B\Big(\frac{\xi}{|\xi|}\Big)\big[\hat{g}(\xi)\big]\Big],\qquad g\in W_1^\Omega. 
}
This definition requires explanations. First, we apply the functional~$B(\zeta)$,~$\zeta = \xi/|\xi|$, to a vector~$\hat{g}(\xi)$, which, by the assumption~$g\in W_1^\Omega$, belongs to the domain of~$B$, that is~$\Omega(\zeta)$. Second, even if the formula is logically correct, it is not clear why the function on the right hand side is even a distribution since its Fourier transform might have a singularity~$|\xi|^{-d}$ at the origin. There are two ways to get around this peculiarity: either regularize the symbol of our multiplier or consider only good functions~$g$ that form a dense set in~$W_1^\Omega$.  While we will follow the former way in Section~\ref{S2}, we provide a short comment about the latter. It is not clear whether the Schwartz functions are dense in~$W_1^\Omega$ when~$\Omega$ is not smooth; for smooth~$\Omega$ this is true, see~\cite{AyoushWojciechowski2017} (though this density statement is not formulated in that paper, the notation and definitions there make the statement almost obvious). However, the set of functions~$g\in W_1^\Omega$ whose spectrum is compact and does not contain the origin, is dense in~$W_1^\Omega$; one may use the formula~\eqref{Functional} for such functions~$g$ only and then pass to the limit whenever an appropriate uniform estimate is proved.
\begin{Th}\label{LinftyTheorem}
Let~$\Omega$ and~$B$ be H\"older continuous and let~$B$ be weakly canceling\textup, that is
\eq{\label{Weaklycanceling}
\forall v\in V\quad \int\limits_{S^{d-1}} B(\zeta)[v]\,d\sigma(\zeta) = 0,\qquad V = \bigcap_{\zeta \in S^{d-1}} \Omega(\zeta).
}
Then\textup, the operator~$\B$ given by~\eqref{Functional} maps~$W_1^\Omega$ to~$L_{\infty}$ continuously.
\end{Th}
Here and in what follows,~$\sigma$ stands for the Lebesgue measure on the unit sphere. We will comment in Section~\ref{S4} how this theorem generalizes Raita's result. Now we only mention that if~$\Omega$ is canceling (satisfies~\eqref{CancelBundle}), then~$V$ is trivial, and any H\"older continuous functional~$B$ is admissible.
\begin{Rem}
By a simple limit argument\textup, the space~$L_{\infty}$ in Theorem~\ref{LinftyTheorem} may be replaced with~$C(\R^d)$. Note that, in the classical setting of differential operators, a similar effect holds for a related~$\BV$ space \textup(defined similarly to~\eqref{BV} below\textup) if and only if~$\Omega$ is canceling\textup, i. e. satisfies~{\eqref{CancelBundle}}\textup, see~\textup{\cite{RaitaSkorobogatova2020}} for details.
\end{Rem}

Let us provide a short comment about the proofs. In fact, the proofs are not longer than either introduction, or statement of the results in higher generality in Section~\ref{S2}, or the applications at the end of the paper. Theorem~\ref{L2Theorem} follows from Theorem~\ref{LinftyTheorem} by simple bilinearization; such a trick was used in~\cite{KMS2015} in a similar context. The proof of Theorem~\ref{LinftyTheorem} is not more difficult: we extend the functional~$B$ to the whole space~$\Co^\ell$ preserving the weak cancellation condition~\eqref{Weaklycanceling} and H\"older continuity and use Mikhlin's principle that a singular integral has bounded Fourier transform. In fact, the proof of Theorem~\ref{LinftyTheorem} mimics the proof of the corresponding theorem in discrete setting in~\cite{Stolyarov2019} (see~\cite{ASW2018} for discrete probabilistic interpretation of canceling operators).

Classical embedding theorems allow a useful generalization that replaces classical homogeneity with respect to isotropic dilations with the so-called anisotropic homogeneity (see, e.g. the book~\cite{BIN1978}). The progress in this direction related to canceling and weakly canceling operators is desirable and seemingly, not straightforward. However, the methods of the present paper are so elementary that we are able to obtain anisotropic versions for Theorems~\ref{L2Theorem} and~\ref{LinftyTheorem}; we present them in Section~\ref{S2}.

\section{Anisotropic homogeneity and anisotropic singular integrals}\label{S2}
Our notation is close to that of~\cite{BesovLizorkin1967}. Let~$a\in \mathbb{R}^d$ be a vector with positive coordinates such that~$\sum_1^d a_j = d$. We call such~$a$ a pattern of homogeneity. Define the~$a$-dilations by the formula
\eq{
\Dil_t(\xi) = \big(t^{-a_1}\xi_1, t^{-a_2}\xi_2,\ldots, t^{-a_d}\xi_d\big),\quad \xi = (\xi_1,\xi_2,\ldots, \xi_d)\in \mathbb{R}^d,\ t > 0.
}
A function~$\Phi\colon\mathbb{R}^d \to \mathbb{C}$ is called~$a$-homogeneous of order~$m$ provided
\eq{
\Phi(\Dil_t(\xi)) = t^{-m}\Phi(\xi)
} 
for any~$\xi\in\mathbb{R}^d$ and any~$t > 0$. This definition also applies to vector-valued functions. Let us define the function~$\eta \colon \R^d\setminus \{0\} \to\R$ implicitly:
\eq{
\sum\limits_{j=1}^d \frac{\xi_j^2}{\eta(\xi)^{2a_j}} = 1.
}
In other words, for any~$\xi$ the trajectory~$t\mapsto \Dil_t(\xi)$ intersects the unit sphere~$S^{d-1}$ at the moment~$t = \eta(\xi)$; the quantity~$\eta(\xi)$ may be regarded as the "norm" of~$\xi$. One may replace~$\eta$ with an equivalent expression that is more convenient for computations:
\eq{\label{AsymptoticEta}
\eta(\xi) \lesssim  \Big(\sum\limits_{j=1}^d|\xi_j|^{\frac{2}{a_j}}\Big)^{\frac12} \lesssim \eta(\xi).
}
The function~$\eta$ gives rise to a "polar change of variables"~$\xi = (\eta, \zeta)$, here~$\zeta = \zeta(\xi) \in S^{d-1}$ is given by
\eq{
\zeta_j = \eta(\xi)^{-a_j}\xi_j,\quad j = 1,2,\ldots,d.
} 
In other words,~$\zeta = \Dil_{\eta(\xi)}(\xi)$.

The polar change of variables reads as follows:
\eq{
\int\limits_{\mathbb{R}^d}\Psi(\xi)\,d\xi = \int\limits_{\mathbb{R}_+} \eta^{d-1}\int\limits_{S^{d-1}} \Big(\sum\limits_{j=1}^d a_j \zeta_j^2\Big) \Psi(\Dil_{\eta^{-1}}(\zeta))\,d\sigma(\zeta)\,d\eta.
}
The proof of the formula relies upon homogeneity considerations and the fact that the scalar product of the unit normal to the sphere~$S^{d-1}$ and the tangent to the trajectory~$t\mapsto \Dil_t(\zeta)$ at their intersection point~$\zeta$ equals exactly~$\sum_{j=1}^d a_j \zeta_j^2$. In particular,
\eq{\label{SubstitutionHom}
\int\limits_{\mathbb{R}^d}\Phi(\xi)\,d\xi = \int\limits_{\mathbb{R}_+} \eta^{m+d-1}\int\limits_{S^{d-1}} \Big(\sum\limits_{j=1}^d a_j \zeta_j^2\Big) \Phi(\zeta)\,d\sigma(\zeta)\,d\eta,
}
provided~$\Phi$ is~$a$-homogeneous of order~$m$. We will need an anisotropic version of Mikhlin's principle that a singular integral has bounded Fourier transform. We will use a version from~\cite{BesovLizorkin1967} (see~\cite{FabesRiviere1966} for a similar theory; we prefer a version from~\cite{BesovLizorkin1967} since it indicates smoothness assumptions on the kernel in a more explicit way; see the short report~\cite{BIL1966} as well).

\begin{Le}\label{SingularIntegral}
Let~$K\colon\mathbb{R}^d\setminus \{0\}\to \Co$ be a continuous function that is $a$-homogeneous of order~$-d$. Assume it satisfies the cancellation condition
\eq{\label{MikhlinCanc}
\int\limits_{S^{d-1}}\Big(\sum\limits_{j=1}^d a_j \zeta_j^2\Big) K(\zeta)\,d\sigma(\zeta) = 0
}
and smoothness condition
\eq{\label{Dini}
\int\limits_{0}^1\frac{w(t)}{t}\,dt < \infty; \quad w(t) = \sup\Set{|K(\zeta_1) - K(\zeta_2)|}{\zeta_1,\zeta_2\in S^{d-1},\ |\zeta_1-\zeta_2|\leq t}.
}
Then\textup, the truncated functions~$K_{\eps,R}$\textup, $0 < \eps< R$\textup, defined by the formula
\eq{
K_{\eps,R}(\xi) = \begin{cases}K(\xi),\quad &\eta(\xi) \in [\eps,R];\\ 0,\quad &\hbox{otherwise},\end{cases}\qquad \xi= (\eta,\zeta),
}
have uniformly bounded Fourier transforms. In particular\textup, the singular integral
\eq{
\varphi \mapsto \vp\int\limits_{\R^d}K(\xi)\varphi(\xi)\,d\xi = \lim\limits_{\genfrac{}{}{0pt}{-2}{\eps \to 0}{R\to \infty}}\int\limits_{\R^d} K_{\eps,R}(\xi)\varphi(\xi)\,d\xi, \quad \varphi \in C_0^{\infty}(\R^d),
}
defines a distribution with bounded Fourier transform.
\end{Le}
\begin{Rem}
There is a little mess in the terminology. As we will see\textup, the classical cancellation condition~\eqref{MikhlinCanc} is more related to weakly cancelling condition~\eqref{Weaklycanceling} rather than Van Schaftingen's cancellation condition~\eqref{Cancel}.
\end{Rem}
\begin{Rem}
The smoothness condition~\eqref{Dini} holds true provided~$K|_{S^{d-1}}$ is H\"older continuous.
\end{Rem}

We transfer the notions from introduction to the anisotropic setting. Let us start with~\eqref{IsotropicOmega}:
\eq{\label{GeneralOmega}
W_1^{\Omega,a} = \Set{f \in L_1(\R^d,\Co^\ell)}{\forall \xi \in \R^d\setminus \{0\}\quad \hat{f}(\xi)\in \Omega(\Dil_{\eta(\xi)}(\xi))}.
}
Note that this space is translation invariant and dilation invariant with respect to~$a$-dilations. We equip it with the norm inherited from of~$L_1(\R^d,\Co^\ell)$. We will suppress the dependence on the pattern of homogeneity in our notation and write simply~$W_1^\Omega$ instead of~$W_1^{\Omega,a}$; the same principle applies to functionals introduced below.

Consider the vectorial Fourier multipliers~$\B_{\eps,R}^a$, where~$0 < \eps < R$, acting on functions~$f\in W_\Omega^1$ by the rule
\eq{\label{AeR}
\B_{\eps,R}^a[f] = \mathcal{F}^{-1}\Big[\chi_{\eta\in [\eps,R]}(\xi) \eta^{-d} B(\Dil_{\eta}(\xi))[\hat{f}(\xi)]\Big],\quad \xi = (\eta,\zeta).
}
Note that the formula is correct since~$\hat{f}(\xi) \in \Omega(\Dil_{\eta}(\xi))$ for any~$\xi$ and thus, the functional~$B(\Dil_{\eta}(\xi))$ may be applied to~$\hat{f}(\xi)\in \Omega(\Dil_{\eta(\xi)}(\xi))$. The version of the operator~$\B^a$ without the cutoff function is~$a$-homogeneous. The cutoff function is needed to make the formula correct (note that~$\B_{\eps,R} [f]$ is a real analytic function).

\begin{Th}\label{RaitaTheorem}
The operator~$\B_{\eps,R}$ is uniformly with respect to~$\eps$ and~$R$ bounded as a mapping of~$W_1^\Omega$ to~$L_{\infty}$ if and only if the symbol~$B$ satisfies the weak cancellation condition
\eq{\label{WeakCancellation}
\int\limits_{S^{d-1}} B(\zeta)[v] \Big(\sum\limits_{j=1}^d a_j \zeta_j^2\Big)\,d\zeta = 0
}
for any~$v \in V$\textup, where
\eq{\label{DefV}
V =  \bigcap\limits_{\zeta \in S^{d-1}} \Omega(\zeta).
}
\end{Th} 
The theorem above implies Theorem~\ref{LinftyTheorem}.
Let us write
\eq{\label{Jacobian}
\J(\zeta) = \sum\limits_{j=1}^d a_j \zeta_j^2
}
for brevity. The lemma below is the core of the matter.
\begin{Le}\label{Extension}
Let~$B$ and~$\Omega$ be the same as defined above. There exists a function~$\tilde{B}\colon S^{d-1}\times \Co^\ell \to \Co$ that is linear with respect to the second variable\textup, H\"older continuous\textup, extends~$B$ in the sense that
\eq{
\forall \zeta \in S^{d-1}, \ v\in \Omega(\zeta)\qquad \tilde{B}(\zeta)[v] = B(\zeta)[v],
}
and satisfies the cancellation condition
\eq{\label{CancellationTilde}
\int\limits_{S^{d-1}} \tilde{B}(\zeta)[v] \J(\zeta)\,d\zeta = 0
}
for any~$v\in \Co^\ell$.
\end{Le}

By~$\I_\beta$ we denote the $a$-homogeneous version of the classical Riesz potential:
\eq{\label{AnisotropicRiesz}
\I_\beta[f]=\mathcal{F}^{-1}\Big[\hat{f}(\xi)\eta^{-\beta}(\xi)\Big],\quad f\in L_1(\R^d),\ \beta \in (0,d).
}
Note that this definition applies to the case~$f$ is a vectorial measure of bounded variation as well (however, the result of application of~$\I_\beta$ might be a distribution).

The theorem below is an anisotropic generalization of Theorem~\ref{L2Theorem}.
\begin{Th}\label{AnisotropicL2}
Assume the function~$\Omega \colon S^{d-1}\to G(\ell,k)$ is H\"older continuous and canceling in the sense that the corresponding space~$V$ \textup(defined in~\eqref{DefV}\textup) is trivial. Then\textup, the $a$-homogeneous Riesz potential~$\I_{\frac{d}{2}}$ maps the~$a$-homogeneous space~$W_1^\Omega$ to~$L_2$.
\end{Th}

\section{Proofs}
\begin{proof}[Proof of Lemma~\ref{Extension}]
Let~$\Omega^\perp(\zeta)$ be the orthogonal complement of~$\Omega(\zeta)$ in~$\Co^\ell$. We will search for~$\tilde{B}$ of the form
\eq{
\tilde{B}(\zeta)[v] = B(\zeta)[\pi_{\Omega(\zeta)}[v]] + \scalprod{b(\zeta)}{\pi_{\Omega^\perp(\zeta)}[v]},
}
where~$b\colon S^{d-1} \to \Co^\ell$ is a vectorial function to be chosen later. Note that this~$\tilde{B}$ extends~$B$ since the second term vanishes for~$v\in \Omega(\zeta)$ and the first equals~$B(\zeta)[v]$ in this case. Moreover, if~$b$ is H\"older continuous, then~$\tilde{B}$ is H\"older continuous as well. 

Note that the linear functional
\eq{
v\mapsto \int\limits_{S^{d-1}} B(\zeta)[\pi_{\Omega(\zeta)}[v]] \J(\zeta)\,d\sigma(\zeta),\qquad v\in \Co^\ell,
}
vanishes on~$V$ (recall~\eqref{WeakCancellation} and~\eqref{DefV}). Therefore, there exists~$a\in V^\perp$ such that
\eq{
\scalprod{a}{u} = -\int\limits_{S^{d-1}} B(\zeta)[\pi_{\Omega(\zeta)}[u]] \J(\zeta)\,d\sigma(\zeta),\qquad u \in\Co^\ell.
}
Thus, in order to fulfill~\eqref{CancellationTilde}, the function~$b$ suffices to satisfy
\eq{\label{SearchForb}
\int\limits_{S^{d-1}}\pi_{\Omega^\perp(\zeta)}[b(\zeta)] \J(\zeta)\,d\sigma(\zeta) = a.
}

Let~$T\colon C^{\infty}(S^{d-1},\Co^\ell)\to \Co^\ell$ be the linear operator that maps~$b$ to
\eq{
\int\limits_{S^{d-1}}\pi_{\Omega^\perp(\zeta)}[b(\zeta)] \J(\zeta)\,d\sigma(\zeta).
}
It suffices to show that the image of~$T$ is dense in~$V^{\perp}$ (in such a case this image coincides with~$V^{\perp}$ and there exists a smooth solution~$b$ to the equation~\eqref{SearchForb}).

First, it is clear that~$\Image[T] \subset V^\perp$ since for any~$v \in V$ we have~$\pi_{\Omega^\perp(\zeta)}[v] = 0$ for any~$\zeta$ and thus,
\eq{
\scalprod{T[b]}{v} = \int\limits_{S^{d-1}}\scalprod{b(\zeta)}{\pi_{\Omega^\perp(\zeta)}[v]} \J(\zeta)\,d\sigma(\zeta) = 0,
}
which means~$T[b]$ is orthogonal to~$V$ for any choice of~$b$.

Second, if the image of~$T$ were not dense in~$V^\perp$, there would have existed~$w\in V^{\perp}$ such that
\eq{
\Scalprod{\int\limits_{S^{d-1}}\pi_{\Omega^\perp(\zeta)}[b(\zeta)] \J(\zeta)\,d\sigma(\zeta)}{w} = 0
}
for any~$b$. This is the same as
\eq{
\int\limits_{S^{d-1}} \scalprod{b(\zeta)}{\pi_{\Omega^\perp(\zeta)}[w]}\J(\zeta)\,d\sigma(\zeta) = 0.
}
Note that the function~$\zeta \mapsto \J(\zeta)\pi_{\Omega^\perp(\zeta)}[w]$ is H\"older continuous. Consequently, by the standard limit argument, the same identity also holds true for any continuous function~$b$, and therefore,~$\pi_{\Omega^\perp(\zeta)}[w]$ vanishes identically. This means~$w = 0$ since~$w\in V^\perp$.
\end{proof}
\begin{proof}[Proof of Theorem~\ref{RaitaTheorem}]
The "only if part". Note that if~$v\in V$, then~$\delta_0\otimes v \in \BV^\Omega$, where~$\delta_0$ is the Dirac delta situated at the origin and the definition of the corresponding~$\BV$ space is similar to~\eqref{GeneralOmega}:
\eq{\label{BV}
\BV^{\Omega,a} = \Set{\mu \in \Me(\R^d,\Co^\ell)}{\forall \xi \in \R^d\setminus \{0\}\quad \hat{\mu}(\xi)\in \Omega(\Dil_{\eta(\xi)}(\xi))};
}
by~$\Me(\R^d,\Co^\ell)$ we denote the space of~$\Co^\ell$-valued measures of bounded variation. The norm of~$\B_{\eps, R}$ as an operator between~$\BV^\Omega$ and~$L_\infty$ equals the norm of the same multiplier treated as an~$W_1^\Omega \to L_{\infty}$ operator. It remains to notice that
\eq{
\B_{\eps,R}[\delta_0\otimes v](0) = \int\limits_{\eps}^R\eta^{-1}\,d\eta \int\limits_{S^{d-1}} B(\zeta)[v] \Big(\sum\limits_{j=1}^d a_j \zeta_j^2\Big)\,d\zeta
}
by virtue of~\eqref{SubstitutionHom}.

To prove the "if" part, apply Lemma~\ref{Extension} and construct the operators~$\tilde{\B}_{\eps,R}$ substituting~$\tilde{B}$ for~$B$ in~\eqref{AeR}. Since~$\tilde{B}$ is an extension of~$B$, we have
\eq{
\B_{\eps,R}[f] = \tilde{\B}_{\eps,R}[f],\qquad f\in W_1^{\Omega}.
}
Therefore, it suffices to show that~$\tilde{\B}_{\eps,R}$ maps~$L_1(\R^d,\Co^\ell)$ to~$L_{\infty}$ (with uniform norm bounds w.r.t.~$\eps$ and~$R$). This is the same as to show that the Fourier transform of the function
\eq{
\xi \mapsto \chi_{\eta\in [\eps,R]}(\xi) \eta^{-d} \tilde{B}(\Dil_{\eta}(\xi))
}
has uniformly w.r.t $\eps$ and~$R$ bounded Fourier transform. This function is vector-valued and~$a$-homogeneous of order~$-d$. So it remains to pair it with an arbitrary vector~$v\in \Co^\ell$ and apply Lemma~\ref{SingularIntegral}. The application of Lemma~\ref{SingularIntegral} is legal by~\eqref{CancellationTilde}.
\end{proof}

\begin{proof}[Proof of Theorem~\ref{AnisotropicL2}]
Pick an arbitrary~$f\in W_1^\Omega$. By the Plancherel theorem, it suffices to prove a uniform with respect to~$f$,~$\eps$, and~$R$ bound
\eq{
\int\limits_{\eta \in [\eps,R]}\frac{|\hat{f}(\xi)|^2}{\eta^d(\xi)}\,d\xi \lesssim \|f\|_{W_1^\Omega}^2.
}
Let us bilinearize this quadratic inequality:
\eq{
\Big|\int\limits_{\eta \in [\eps,R]}\eta^{-d}(\xi) \scalprod{\hat{f}(\xi)}{\hat{g}(\xi)}_{\Co^\ell}\,d\xi\Big| \lesssim \|f\|_{W_1^\Omega}\|g\|_{L_1(\R^d,\Co^\ell)}.
}
The expression on the left hand side might be represented as
\eq{
\Big|\sum\limits_{j=1}^\ell \int\limits_{\R^d}  \B_{j,\eps,R}[f](x)\overline{g_j(x)}\,dx \Big|, 
}
where~$\B_{j,\eps,R}$ is the Fourier multiplier with the symbol~\eqref{AeR}, where the corresponding functional~$B$ maps~$v\in\Omega(\zeta)$ to its~$j$-th coordinate~$v_j$ (in particular,~$B$ does not depend on~$\zeta$). Since~$V = \{0\}$, the operators~$\B_j$ fulfill the requirements of Theorem~\ref{RaitaTheorem}, and thus,
\eq{
\|\B_{j,\eps,R}[f]\|_{L_{\infty}} \lesssim \|f\|_{W_1^\Omega}
}
uniformly with respect to~$j, \eps, R$, and~$f$, and the theorem is proved.
\end{proof}

\section{Examples}\label{S4}
\paragraph{Theorem~\ref{LinftyTheorem} and differential operators.} Let us rewrite condition~\eqref{Weaklycanceling} in terms of differential operators. Let~$A$ be an elliptic differential operator that maps~$V$-valued functions to~$E$-valued functions and is (isotropic) homogeneous of order~$m \geq d$. We are interested in the inequality
\eq{\label{Interest}
\|P f\|_{L_{\infty}} \lesssim \|A f\|_{L_1},\quad f\in C_0^{\infty}(\R^d,V),
}
here~$P$ is a homogeneous differential operator of order~$m-d$. We choose
\eq{
\Omega(\zeta) = \Image (A(\zeta));\quad B(\zeta) = P(\zeta)(A^*(\zeta)A(\zeta))^{-1}A^*(\zeta),
}
where~$P$ is identified with its symbol. Note that 
\eq{
\B[A f] = Pf,\quad f\in C_0^{\infty}(\R^d,V),
}
so~\eqref{Interest} is equivalent to
\eq{
\|\B g\|_{L_{\infty}} \lesssim \|g\|_{W_1^\Omega}.
}
By Theorem~\ref{LinftyTheorem}, the latter inequality holds true if and only if
\eq{
\forall v\in \bigcap_{\zeta \in S^{d-1}}\Image(A(\zeta))\qquad \int\limits_{S^{d-1}} P(\zeta)(A^*(\zeta)A(\zeta))^{-1}A^*(\zeta)[v]\,d\sigma(\zeta) = 0;
}
this coincides with the weak cancellation condition in~\cite{Raita2019}.

\paragraph{Embedding theorems from~\cite{KMS2015}.} The paper~\cite{KMS2015} concerned the isomorphism problem for the Banach spaces~$\BV^A$ and their relatives. Embedding theorems serve as useful tools for such results (this goes back to~\cite{Kislyakov1975}).  We will show how the embedding theorems from~\cite{KMS2015} may be derived from Theorem~\ref{AnisotropicL2} on two examples (we do not dwell on an interesting effect of transference of Fourier coefficients present in the most general embedding theorem in~\cite{KMS2015}).

Consider the inequality
\eq{\label{PreGN}
\|f\|_{L_2(\R^3)} \lesssim \Big\|\frac{\partial f}{\partial x_1}\Big\|_{L_1(\R^3)}\|\Delta_{x_2,x_3} f\|_{L_1(\R^3)}
}
for smooth compactly supported functions on~$\R^3$; this inequality was proved in~\cite{KMS2015}. The operator~$\Delta_{x_2, x_3}$ is the Laplacian with respect to the variables~$x_2$ and~$x_3$. By homogeneity considerations,~\eqref{PreGN} is equivalent to
\eq{\label{GN}
\|f\|_{L_2(\R^3)} \lesssim \Big\|\frac{\partial f}{\partial x_1}\Big\|_{L_1(\R^3)} + \|\Delta_{x_2,x_3} f\|_{L_1(\R^3)}.
}
Consider the pattern of homogeneity~$a = (\frac32,\frac34,\frac34)$, let~$\ell = 2$,  and define the function~$\Omega \colon S^2\to G(2,1)$ by the rule
\eq{
S^2\ni \zeta = (\zeta_1,\zeta_2,\zeta_3)\ \longmapsto\ \Omega(\zeta) = \Co\cdot(\zeta_1,\zeta_2^2+\zeta_3^2) \subset \Co^2.
}
This function is smooth and satisfies the cancellation condition~$V = \{0\}$ (the space~$V$ is defined in~\eqref{DefV}). Therefore, we may apply Theorem~\ref{AnisotropicL2}. Using formula~\eqref{AsymptoticEta}, we obtain 
\eq{
\eta^{-\frac32}(\xi) \asymp \Big(|\xi_1|^{\frac43} + |\xi_2|^{\frac83} + |\xi_3|^{\frac83}\Big)^{-\frac34},
}
and Theorem~\ref{AnisotropicL2} leads to
\mlt{
\Big\|\Big(|\xi_1|^{\frac43} + |\xi_2|^{\frac83} + |\xi_3|^{\frac83}\Big)^{-\frac34} \xi_1\hat{f}(\xi)\Big\|_{L_2(\R^3)} + \Big\|\Big(|\xi_1|^{\frac43} + |\xi_2|^{\frac83} + |\xi_3|^{\frac83}\Big)^{-\frac34} (\xi_2^2 + \xi_3^2) \hat{f}(\xi)\Big\|_{L_2(\R^3)} \lesssim\\ \Big\|\frac{\partial f}{\partial x_1}\Big\|_{L_1(\R^3)} + \|\Delta_{x_2,x_3} f\|_{L_1(\R^3)},
}
which implies~\eqref{GN}.

Let us cite Theorem $0.2$ in~\cite{KMS2015}. In that theorem,~$\kappa$ and~$\lambda$ were integer numbers,~$\alpha$ and~$\beta$ were non-negative reals, and
\begin{equation}\label{sobolevspace}
W_2^{\alpha,\beta}(\mathbb{R}^2)=
\{f \in\mathcal{S}'(\mathbb{R}^2) \mid
|\xi|^{\alpha} |\eta|^{\beta}\hat{f}(\xi,\eta) \in
L^2(\mathbb{R}^2)\};
\end{equation}
there are small inaccuracies in this definition, see~\cite{KMS2015}; by~$\mathcal{S}'$ we denote the class of Schwartz distributions on~$\R^2$.

\begin{Th}[Theorem~$0.2$ in~\cite{KMS2015}]\label{KMSMain}
Let~$\varphi_j$\textup,~$j=1,\dots,N$\textup, be compactly supported distributions
on the plane~$\mathbb{R}^2$. Suppose that they satisfy the equations
\begin{equation}\label{syst}
\left\{
\begin{aligned}
-\partial_1^\kappa \varphi_1 &&&=  &\mu_0& ;\\
-\partial_1^\kappa \varphi_2 &+\partial_2^\lambda \varphi_1&& =  &\mu_1& ;\\
&\vdots&&   = &\,\vdots\,& ;\\
-\partial_1^\kappa\varphi_j &+\partial_2^\lambda \varphi_{j-1}&& = &\mu_{j-1}& ;\\
&\vdots&&   = &\,\vdots\,& ;\\
-\partial_1^\kappa \varphi_N &+\partial_2^\lambda \varphi_{N-1}&& = &\mu_{N-1}& ;\\
 &\phantom{+}\;\,\partial_2^\lambda \varphi_N&& = &\mu_N&,
\end{aligned}
\right.
\end{equation}
where~$\mu_0,\dots,\mu_N$ are finite \textup{(}compactly
supported\textup{)} measures on the plane. The inequality
\begin{equation}\label{est1}
\sum\limits_{j = 1}^{N}\|\varphi_j\|_{W_2^{\alpha,\beta}(\mathbb{R}^2)}
\lesssim \sum\limits_{j = 0}^{N}\mathrm{var}\mu_j
\end{equation}
holds true whenever~$\alpha$ and~$\beta$ are nonnegative and satisfy
\begin{equation}\label{line}
\frac{\alpha + \frac{1}{2}}{\kappa} + \frac{\beta + \frac{1}{2}}{\lambda} = 1.
\end{equation}
\end{Th}
Let us deduce this theorem from Theorem~\ref{AnisotropicL2}. For that, set~$a=\big(\frac{2\kappa}{\kappa+\lambda},\frac{2\lambda}{\kappa+\lambda}\big)$,~$\ell = N+1$,~$k = N$, and
\eq{
\Omega(\zeta) = \Set{v\in \Co^{N+1}}{\sum\limits_{j=0}^N \zeta_1^{j \kappa}\zeta_2^{(N+1-j)\lambda}v_j = 0},\qquad \zeta \in S^1.
}
This~$\Omega$ satisfies the cancellation condition. Note that by~\eqref{AsymptoticEta},
\eq{
\eta^{-1}(\xi) \asymp \Big(|\xi_1|^{\frac{\kappa + \lambda}{\kappa}} + |\xi_2|^{\frac{\kappa + \lambda}{\lambda}}\Big)^{-\frac12},
}
and Theorem~\ref{AnisotropicL2} (modulo the equivalence of embeddings of~$W_1^\Omega$ and~$\BV_1^\Omega$, which is a routine approximation argument) implies Theorem~\ref{KMSMain} similar to the previous case.

Embedding theorems in the style of Theorem~\ref{KMSMain} for more general systems of equations were later obtained in~\cite{KislyakovMaximov2018}. These theorems also follow from Theorem~\ref{AnisotropicL2} (however, this provides no simplification against the original reasoning in~\cite{KislyakovMaximov2018}).

\paragraph{Comment on bilinear estimates in~\cite{Stolyarov2015}.} The paper~\cite{Stolyarov2015} concerned bilinear estimates similar to that used in~\cite{KMS2015}. Namely, the inequalities
\begin{equation}\label{BilinearEmbeddings}
\left|\langle f,g\rangle_{W_2^{\alpha,\beta}}\right| \lesssim \left\|(\partial_1^\kappa -\tau\partial_2^\lambda)f\right\|_{L_1}\left\|(\partial_1^\kappa -\sigma\partial_2^\lambda)g\right\|_{L_1};\quad \partial_1 = \frac{\partial}{\partial x}, \partial_2 = \frac{\partial}{\partial y},\ f,g\in C_0^{\infty}(\R^2),
\end{equation}
were investigated, the Hilbert space on the left hand side is the same as in~\eqref{sobolevspace}. Under the assumption that the operators on the right are elliptic, that paper provided simple if and only if conditions on the parameters~$\sigma,\tau \in \Co$,~$\alpha,\beta \geq 0$, when~\eqref{BilinearEmbeddings} holds true. It appears this result is closely related to Lemma~\ref{SingularIntegral}. What is more, Lemma~\ref{SingularIntegral} allows to work in a more general setting.

\begin{St}\label{BilinearProp}
Let~$P_1$ and~$P_2$ be elliptic~$a$-homogeneous scalar differential operators in~$d$ variables of orders~$d_1$ and~$d_2$ correspondingly. Let~$Q$ be a scalar Fourier multiplier with H\"older continuous symbol that is~$a$-homogeneous of order~$d_1 + d_2 - d$. The inequality
\eq{
\Big|\int\limits_{\R^d}Qf(x) \overline{g(x)}\,dx\Big|\lesssim \|P_1 f\|_{L_1(\R^d)} \|P_2 g\|_{L_1(\R^d)}
}
holds true for any smooth compactly supported functions~$f$ and~$g$ with uniform constant if and only if
\eq{
\int\limits_{\zeta \in S^{d-1}} \frac{Q(\zeta) \J(\zeta)}{P_1(\zeta)\overline{P_2(\zeta)}}\,d\sigma(\zeta) = 0,\qquad \J\hbox{is defined in~\eqref{Jacobian}}.
}
\end{St}
In the formula above, we identify the differential operators~$P_i$ with their symbols; the same principle applies to~$Q$. Proposition~\ref{BilinearProp} is an immediate consequence of Lemma~\ref{SingularIntegral}. The integral that appears if apply this proposition to inequality~\eqref{BilinearEmbeddings} looks like this:
\eq{
\int\limits_{0}^{2\pi}\frac{|\cos \varphi|^{2\alpha}|\sin\varphi|^{2\beta}(\frac{\cos^2\varphi}{\kappa} + \frac{\sin^2\varphi}{\lambda})\,d\varphi}{(\cos^\kappa \varphi -\tau_1\sin^\lambda\varphi)(\cos^\kappa \varphi -\sigma_1\sin^\lambda\varphi)},
}
where~$\tau_1 = (2\pi i)^{\lambda - \kappa}\tau$,~$\sigma_1 = (-1)^{\lambda - \kappa}(2\pi i)^{\lambda - \kappa}\overline{\sigma}$. The latter integral is proportional to the integral
\eq{
\int\limits_{-\infty}^{\infty}\frac{|\rho|^{2\alpha}\,d\rho}{\big(\rho^k - \tau_1\big)\big(\rho^k - \sigma_1\big)} = 0.
} 
via the change of variables~$\rho = \cos\varphi \sin^{-\frac{\kappa}{\lambda}}\varphi$ and~\eqref{line}, which is, by homogeneity, necessary for~\eqref{BilinearEmbeddings}. This integral is~$(2.6)$ in~\cite{Stolyarov2015}; see~\cite{Stolyarov2015} why the vanishing of this integral is equivalent to the condition that one of the numbers~$\kappa$ and~$\lambda$ is odd,~$\alpha = \frac{\kappa-1}{2}$\textup,~$\beta=\frac{\lambda-1}{2}$\textup, and the numbers~$\sigma_1$ and~$\tau_1$ have non-zero imaginary parts of the same sign; these reasonings are based upon elementary complex variable techniques: the residue formula and the uniqueness theorem for analytic functions.

Inequality~\eqref{BilinearEmbeddings} remains true in some cases where the operators on the right are non-elliptic, i.e. when~$\sigma_1$ and~$\tau_1$ are real; see~\cite{Stolyarov2015}. This gives hope that Proposition~\ref{BilinearProp} and Lemma~\ref{SingularIntegral} may be generalized to the non-elliptic setting (by ellipticity in Lemma~\ref{SingularIntegral} we mean the smoothness condition~\eqref{Dini}, it seems that in some cases~$K|_{S^{d-1}}$ might be a distribution).

\bibliography{mybib}{}
\bibliographystyle{amsplain}

St. Petersburg State University, Department of Mathematics and Computer Science;


d.m.stolyarov at spbu dot ru.
\end{document}